\documentclass[12pt]{amsart}
\usepackage{amsfonts,amsmath,latexsym,amssymb,verbatim,amsbsy,amsthm}

\usepackage{pstricks}
\usepackage{graphicx}

\usepackage[top=1in, bottom=1in, left=0.8in, right=0.8in]{geometry}

\usepackage{enumitem}

\theoremstyle{plain}
\newtheorem{THEOREM}{Theorem}[section]

\newtheorem{PROP}[THEOREM]{Proposition}

\newtheorem{theorem}[THEOREM]{Theorem}

\newtheorem{lemma}[THEOREM]{Lemma}

\theoremstyle{definition}

\theoremstyle{remark}

\newtheorem{remark}[THEOREM]{Remark}

\def \d {\delta}
\def \D {\Delta}
\def \g {\gamma}

\def \f {\varphi}

\def \l {\lambda}

\def \n {\nabla}

\def \cF {\mathcal{F}}

\newcommand{\N}{\ensuremath{\mathbb{N}}}   %%% naturals
\newcommand{\Z}{\ensuremath{\mathbb{Z}}}   %%% integers
\newcommand{\R}{\ensuremath{\mathbb{R}}}   %%% reals
\newcommand{\T}{\ensuremath{\mathbb{T}}}   %%% torus

\def \lan {\langle}
\def \ran {\rangle}
\def \p {\partial}
\def \ra {\rightarrow}

\DeclareMathOperator{\diver}{div} %
\begin{document}

\title[Energy Balance for Density-Dependent NSE]{The Energy Balance Relation for Weak solutions of the Density-Dependent Navier-Stokes Equations}
\author{T.M. Leslie}

\author{R. Shvydkoy}

\email{tlesli2@uic.edu; shvydkoy@uic.edu}

\address{University of Illinois at Chicago, Chicago, IL, 60607}
 
 \thanks{This work was partially supported by NSF grants DMS-1210896 and DMS-1515705}

\begin{abstract}
We consider the incompressible inhomogeneous Navier-Stokes equations with constant viscosity coefficient and density which is bounded and bounded away from zero. We show that the energy balance relation for this system holds for weak solutions if the velocity, density, and pressure belong to a range Besov spaces of smoothness $1/3$. A density-dependent version of the classical K\'arm\'an-Howarth-Monin relation is derived.
\end{abstract}


\maketitle

\section{Introduction}
Consider the density-dependent incompressible Navier-Stokes equations:
\begin{equation}
\label{eq:momentum}
\p_t (\rho u) + \diver(\rho u\otimes u) - \mu\Delta u = -\n p + \rho f,
\end{equation}
\begin{equation}
\label{eq:density}
\p_t \rho + \diver (\rho u) = 0,
\end{equation}
\begin{equation}
\label{eq:incompressibility}
\n \cdot u = 0.
\end{equation}
Here $u(x,t)$ represents the $d$-dimensional velocity, $f(x,t)$ is an external force (with values in $\R^d$), $p(x,t)$ is the pressure, $\rho(x,t)$ is the density, and $\mu$ is the viscosity coefficient (which we take to be constant).  We consider \eqref{eq:momentum}-\eqref{eq:incompressibility} for $x\in \T^d$ and $t \geq 0$. It is known, see \cite{Lions,Huang/Wang,Simon}, that if $u_0$ is divergence-free and square-integrable, $\underline{\rho}\le \rho_0 \le \overline{\rho}$ for some positive constants $\underline{\rho}$ and $\overline{\rho}$, and if $f\in L^2([0,T];L^2(\T^d))$, then there exists a Leray-Hopf type global weak solution to the system $(\rho, u)$ such that $\underline{\rho}\le \rho \le \overline{\rho}$,  $u\in L^2([0,T];H^1(\T^d))$, and $(\rho, u)$ satisfies the energy inequality
\begin{equation}
\label{eq:energy_inbalance}
E(t) - E(0) \leq - \mu \int_0^t \|\n u\|_{L^2(\T^d)}^2\,ds + \int_0^t \int_{\T^d} \rho u \cdot f \,dx\,ds,
\quad \text{ where } E(s) = \frac12\int_{\T^d\times \{s\}} \rho|u|^2\,dx.
\end{equation}
Fluids with variable distribution of density arise in many physical contexts. In particular, they appear prominently in Rayleigh-Taylor mixing when a heavier layer fluid on top of lighter one gets mixed under the force of gravity, creating an non-homogeneous turbulent layer. Although an analogue of the classical Kolmogorov theory of turbulence for non-homogeneous fluids has not yet been developed, it appears to be evident that under proper self-similarity assumptions on the velocity increments $\d u = u(r+\ell) - u(r)$ and density $\d\rho$ a limited level of regularity would be expected of $u$ and $\rho$ in the limit of vanishing viscosity. Such regularity should allow for a residual amount of energy to be dissipated in the limit by analogy with the Kolmogorov's $0$th law of turbulence, see \cite{Frisch}. Mathematical study of the question of what this critical regularity might be has been a subject of many recent publications centered around the so-called Onsager conjecture, which states that for the pure Euler equation H\"older exponent $1/3$ gives a threshold regularity between energy conservation and existence of dissipative solutions that do not conserve energy (see \cite{eyink-exam,cet,CCFS,Onsager}). In this paper we address the same question in the context of the full density-dependent forced system \eqref{eq:momentum}-\eqref{eq:incompressibility} with or without viscosity. 

Let us recall that a weak solution to \eqref{eq:momentum}-\eqref{eq:incompressibility} is a triple $(\rho, u, p)\in L^\infty_{t,x} \times L^2_{t,x} \times \mathcal{D}'$ ($\mathcal{D}'$ is the space of distributions) such that for any triple of smooth test functions $(\eta, \psi, \gamma)$, one has
\begin{equation}
\label{eq:wk_momentum}
\begin{split}
\int_{\T^d \times \{s\}} \rho u\cdot \psi\,dx \bigg|^t_0 - \int_0^t \int_{\T^d} \big( \rho u\cdot \p_s \psi + & (\rho u\otimes u):\n \psi +  p\diver \psi \big)\,dx\,ds \\ & = \mu \int_0^t \int_{\T^d} u \cdot \Delta \psi\,dx\,ds +\int_0^t \int_{\T^d} \rho f\cdot \psi\,dx\,ds,
\end{split}
\end{equation}
\begin{equation}
\label{eq:wk_density}
\int_{\T^d\times \{s\}} \rho \eta\,dx \bigg|^t_0  = \int_0^t \int_{\T^d} \left( \rho \p_s \eta +  (\rho u\cdot \n) \eta \right)\,dx\,ds,
\end{equation}
\begin{equation}
\label{eq:wk_incompressibility}
\int_{\T^d} u\cdot \n \g = 0.
\end{equation}
In \eqref{eq:wk_momentum}, we write $A:B$ for $\sum_{i,j=1}^d a_{ij}b_{ij}$, where $A=(a_{ij})$, $B=(b_{ij})$. If $\rho$ and $u$ are smooth, then using  $\psi = u$ we readily obtain the energy balance relation:
\begin{equation}
\label{eq:energy_balance}
E(t) - E(0) = - \mu \int_0^t \|\n u\|_{L^2(\T^d)}^2\,ds + \int_0^t \int_{\T^d} \rho u \cdot f \,dx\,ds.
\end{equation}
In the context of weak solutions even in the class $u\in L^2 H^1$, such a manipulation is not feasible due to lack of sufficient regularity to integrate by parts. This leaves room for additional mechanisms of energy dissipation due to the work of the nonlinear term. In the case $\mu = 0$, due to time reversibility the energy may also increase above the legitimate change resulting from the work of force.  Our main result provides a sharp sufficient regularity condition on $(\rho,u,p)$ to guarantee energy balance \eqref{eq:energy_balance} to hold. We use Besov spaces to state our criteria as motivated by numerous previous studies on Onsager conjecture, \cite{cet,CCFS,eyink-besov}. The definitions are standard and recalled in Section~\ref{s:besov}.

\begin{theorem}
\label{t:energy_bal}
Let $(\rho, u, p)$ be a weak solution to the density-dependent incompressible Navier-Stokes equations on $\T^d$, $d>1$.  Assume $(\rho, u,p)$ satisfies
\begin{align}
&u\in L^2([0,T];H^1(\T^d)), \ 0<\underline{\rho}\le \rho \le \overline{\rho}<\infty, \text{ and } f\in L^2([0,T] \times \T^d), \label{hyp:Onsager1} \\
&\rho\in L^a([0,T];B_{a,\infty}^{\frac13}),\;
u\in L^b([0,T]; B_{b,c_0}^{\frac13}),\;
p\in L^{\frac{b}{2}}([0,T]; B^{\frac13}_{\frac{b}{2},\infty}),\quad \frac1a + \frac3b=1,\; b \ge 3.  \label{hyp:Onsager2}
\end{align}
Then $(\rho, u, p)$ satisfies the energy balance relation \eqref{eq:energy_balance} on the time interval $[0,T]$. 
\end{theorem}

The assumption on the pressure in \eqref{hyp:Onsager2}  is in natural correspondence to the condition on velocity. In fact, it follows from the latter in the case of constant density (see Remark~\ref{r:pressure}). Such a conclusion, however, cannot be made in the density dependent case when the density has limited regularity as ours. In general the pressure is only known to exist as a distribution. As the proof goes we will see that the first line of assumptions \eqref{hyp:Onsager1} pertains to the control of the viscous and force terms in the local energy budget relation, while \eqref{hyp:Onsager2} is used to control anomalous flux due to the transport term. So, as a byproduct of the proof we obtain energy conservation condition for the Euler equation.

\begin{theorem}\label{t:EE_result} Suppose $(\rho, u, p)$ is a weak solution to the density-dependent incompressible Euler
equations on $\T^d$ with zero force, the same set of assumptions \eqref{hyp:Onsager2}, and $0<\underline{\rho}\le \rho\le \overline{\rho}<\infty$. Then the energy is conserved in time.
\end{theorem}

In the case when $b = 3$, we obtain the earlier obtained results in the homogeneous case, see \cite{CCFS}. However, in this case one must assume a rather strong regularity on the density: $\rho \in B^{1/3}_{\infty,\infty} = C^{1/3}$, the usual H\"older class. It is shown in \cite{eyink-exam} that Besov space $u\in B^{1/3}_{3,\infty}$ is sharp to control the energy flux in homogeneous fluid. It is therefore not expected to be improved in the above results. We also derive an extension to the density-dependent case of the classical 
K\'arm\'an-Howarth-Monin relation for the energy flux due to nonlinearity in the statistically homogeneous turbulence. It suggests that any of the conditions in the range of \eqref{hyp:Onsager2} arise naturally.

\section{Preliminaries and Preparations for the Main Theorem}

In \cite{CCFS} it was shown that if $u\in L^3([0,T]; B^{1/3}_{3,c_0}(\R^d))\cap C_w([0,T]; L^2(\R^d))$ is a weak solution to the (homogeneous) incompressible Euler equations, then $u$ conserves energy.  %(Here $B^{1/3}_{3,c_0}$ denotes a certain Besov space, described below.)  
The authors define an energy flux $\Pi_Q(t)$ describing the energy dissipated from scales associated to wave numbers $\l_q=2^q$ for $-1\le q\le Q$. To prove their result, they bound $\Pi_Q(t)$ using the convolution of a sequence involving the Littlewood-Paley projections of the solution $u$ with a localization kernel; they conclude by noting that their bound tends to zero in the limit.  We follow a similar program in this section.  After motivating our use of Besov spaces by generalizing the K\'arm\'an-Howarth-Monin relation to the present context, we recall the definition of a Besov space and set some notation.  Next, we derive an energy budget relation associated to the density-dependent Navier-Stokes equations.  Finally, we define localization kernels and present some estimates that will streamline the proof of our theorem.

\subsection{K\'arm\'an-Howarth-Monin relation}

Let us motivate the use of Besov spaces and the choice of regularity classes by ideas from the turbulence theory. Our immediate goal is to extend the classical K\'arm\'an-Howarth-Monin relation to the density-dependent case, see \cite{Frisch}. Let us suppose that our fluid reached a state of fully developed turbulence in which statistical laws with respect to an ensemble average $\lan \cdot \ran$ are independent of a location in space where are measured\footnote{The common term \emph{homogeneous turbulence} may be misleading in our settings as our density still remains variable.}. In order to measure how much regularity is needed to control the energy flux we derive a formula for the physical space energy flux due to the nonlinear transport term defined by
\[
\pi(\ell) = \frac14 \p_t \lan u(r+\ell) \cdot u(r)(\rho(r+\ell) + \rho(r)) \ran_{\mathrm{T}}.
\]
Note that it coincides with the classical flux in the case when $\rho$ is constant, and it is symmetric with respect to $r+\ell, r$. Let us use the notation $u_i = u_i(r)$, $u_i' = u_i(r+\ell)$, $\p_i = \frac{\p}{\p r_i}$, $\p'_i = \frac{\p}{\p \ell_i}$. From the transport term in the momentum equation \eqref{eq:momentum} we obtain
\begin{equation}\label{}
\begin{split}
- 4\pi(\ell)  = \lan \p_j( \rho' u_j' u_i' )u_i \ran + \lan \rho' u_i' \p_j( u_j u_i) \ran + \lan \p_j( \rho u_j u_i )u'_i \ran + \lan \rho u_i \p_j( u'_j u'_i) \ran .
\end{split}
\end{equation}
Note that $\p_j( \rho' u_j' u_i' ) = \p'_j( \rho' u_j' u_i' )$, and 
$\lan \p'_j( \rho' u_j' u_i' ) u_i \ran = \p'_j \lan \rho' u_j' u_i' u_i \ran$. Similarly, $\lan \rho u_i \p_j( u'_j u'_i) \ran = \p'_j \lan \rho u_i u'_j u'_i \ran$. As to the two terms in the middle we first perform integration by parts. This can be justified by first averaging over the fluid domain $\T^d$. Since the ensembles are independent of $r$, this does not change the quantities. Then switching the order of averaging, integrating by parts, switching again, and un-averaging produces the result. So, 
$\lan \rho' u_i' \p_j( u_j u_i) \ran = - \lan \p_j(\rho' u_i') u_j u_i \ran = -\p_j' \lan \rho' u_i' u_j u_i \ran$, and similarly,
$\lan \p_j( \rho u_j u_i )u'_i \ran  = - \lan \rho u_j u_i \p_j(u'_i) \ran  = - \p'_j \lan \rho u_j u_i u'_i \ran$. We thus obtain
\begin{equation}\label{}
4\pi(\ell) = - \p'_j \lan \rho' u_j' u_i' u_i \ran +\p_j' \lan \rho' u_i' u_j u_i \ran + \p'_j \lan \rho u_j u_i u'_i \ran - \p'_j \lan \rho u_i u'_j u'_i \ran.
\end{equation}
Let us denote $\d u(\ell) = u(r+\ell) - u(r)$, and similar for $\rho$. The expression on the right can be shown to equal
%\[
%-\n_\ell \cdot \lan \d \rho \d u (u\cdot \d u) \ran - \n_\ell \cdot \lan \d \rho |\d u|^2 \d u \ran-\n_\ell \cdot \lan \rho |\d u|^2 \d u\ran.
%\]
\begin{equation}
\label{eq:KHM_gradl}
-\n_\ell \cdot \lan (\d(\rho u) \cdot \d u) \d u\ran.
\end{equation}
This can be proved directly by breaking the above into individual terms and noting that $\lan \rho' u_j' u_i'u_i' \ran = \lan \rho u_j u_iu_i \ran$ are independent of $\ell$, and $\p'_j \lan \rho u_i u_i u'_j \ran = 0$ by the divergence-free condition, and  $\p'_j \lan \rho' u_i' u_i' u_j \ran = 0$ by the same reason after changing $r \to r-\ell$.  Applying the algebraic identity
$\delta(fg) = \frac12[(f+f')\d g + (g+g')\d f]$ to \eqref{eq:KHM_gradl}, we obtain
%Finally, let us note that by homogeneity $\pi(\ell) = \pi(-\ell)$. Performing now the same computation with $-\ell$, and making change in the obtained averages from $r$ to $r+\ell$ we obtain
%\[
%4\pi(\ell) = -\n_\ell \cdot \lan \d \rho \d u (u(r+\ell)\cdot \d u) \ran +\n_\ell \cdot \lan \d \rho |\d u|^2 \d u \ran-\n_\ell \cdot \lan \rho(r+\ell) |\d u|^2 \d u\ran.
%\]
%Averaging the two  obtained formulas, the middle terms cancel and we arrive at
\begin{equation}\label{}
\pi(\ell) = -\frac18 \n_\ell \cdot \lan \d \rho \d u ((u(r+\ell)+u(r))\cdot \d u) \ran - \frac18 \n_\ell \cdot \lan (\rho(r+\ell) +\rho(r))|\d u|^2 \d u\ran.
\end{equation}

This is a direct generalization of the classical K\'arm\'an-Howarth-Monin relation. We note that is it seen from this relation that in order for the flux to vanish there are a few possibilities in terms of distribution of smoothness and integrability between $\rho$ and $u$. Given that $\rho \in L^\infty$ is a natural assumption, the last term vanishes if $u$ is $1/3$ regular in $L^3$-sense. Then for the first term to vanish one must also have $u$ being $1/3$ regular in $L^b$-sense and $\rho$ being $1/3$ regular in $L^a$-sense, where $\frac1a + \frac3b = 1$. This leads to the use of Besov spaces and suggests that the set of assumptions \eqref{hyp:Onsager2} is sharp.

\subsection{Besov spaces via Littlewood-Paley decomposition}\label{s:besov}
We follow the setup of \cite{Chesk/Shvyd} and \cite{CCFS} in defining the Littlewood-Paley projections of the functions $\rho, u, p$.  Fix $\chi\in C_0^\infty(B(0,1))$ such that $\chi(\xi)=1$ for $|\xi|\le \frac{1}{2}$. Define $\phi(\xi) = \chi(\frac{\xi}{2}) - \chi(\xi)$.  Define length scales $\l_q = 2^q$, and define $\f_{-1}(\xi) = \chi(\xi)$, $\f_q(\xi) = \phi(\l_q^{-1} \xi)$ for $q\in \N\cup\{0\}$. Then $\sum_{q=-1}^\infty \f_q \equiv 1$; in particular $\sum_{q=-1}^\infty \f_q(k) = 1$ for all $k\in \Z^d$.  We do not distinguish notationally between $\f_q$ and its restriction to the integer lattice, but occasionally it will be necessary to interpret $\f_q$ in the latter sense.  Note that $\f_q$, $\f_r$ have disjoint supports unless $r\in \{q-1, q, q+1\}$.  Let $\cF$ and $\cF^{-1}$ denote the Fourier transform and inverse transform for $\T^d$: $\cF(f)(k) = \int_{\T^d} f(x)e^{-2\pi i k\cdot x}\,dx$, $\cF^{-1}(g)(x) = \sum_{k\in \Z^d} g(k)e^{2\pi i k\cdot x}\,dx$.  

Define the following functions:
\begin{align*}
& h_q = \cF^{-1}(\f_q),
\;\;\;\;\;
\widetilde{h}_Q=\cF^{-1}(\chi(\l_{Q+1}\,\cdot\,)), \\
& u_q = \cF^{-1}(\f_q\cF u)= h_q*u,
\;\;\;\;\; u_{\le Q} = \sum_{q=-1}^Q u_q = \cF^{-1}(\chi(\l^{-1}_{Q+1}\,\cdot\,)\cF u) = \widetilde{h}_Q*u,\\
& u_{\sim Q} = \sum_{q=Q-2}^{Q+2} u_q,
\;\;\;\;\;
u_{>Q} = \sum_{q=Q+1}^\infty u_q.
\end{align*}

Write $A:=\N \cup \{0,-1\}$. The Besov space $B^s_{p,r}(\T^d)$ ($s\in \R$, $p,r\in [1,\infty]$) is the space of tempered distributions $u$ whose corresponding norm, defined by
\[
\|u\|_{B^s_{p,r}(\T^d)} = \left\| (\l_q^s \|u_q\|_{L^p(\T^d)})_{q\in A} \right\|_{\ell^r(A)},
\]
is finite. Clearly $B^s_{p,r}(\T^d)\subset B^{s'}_{p',r'}(\T^d)$ for $s'\le s$, $p'\le p$, $r'\ge r$. Furthermore, $B_{a,\infty}^s\subset L^a$ for all $a\in [1,\infty)$, $s>0$.
%The Bernstein inequalities
%\begin{equation}
%\label{eq:Bernstein_ineq}
%\|u_q\|_{L^b(\T^d)} \lesssim \l_q^{d\left(\frac{1}{a} - \frac{1}{b}\right)} \|u_q\|_{L^a(\T^d)}, 
%\;\;\;\;\;
%(b\ge a\ge 1)
%\end{equation}
%give us additionally (for $b\ge a\ge 1$) the continuous embeddings
%\begin{align}
%\label{eq:Besov_emb1}
%B^s_{a,r}(\T^d) & \subset B_{b,r}^{s-d\left(\frac{1}{a} - \frac{1}{b}\right) }(\T^d), \\
%\label{eq:Besov_emb2}
%B^0_{a,2}(\T^d) & \subset L^a(\T^d), \;\;(a\ge 2).
%\end{align}
We define $B^s_{p,c_0}(\T^d)$ to be the space of tempered distributions $u$ such that $\l_q^s\|u_q\|_{L^p(\T^d)} \stackrel{q\to \infty}{\longrightarrow}0$, together with the norm inherited from $B^s_{p,\infty}(\T^d)$.  Note that this space contains $B^s_{p,r}(\T^d)$ for all $r\in [1,\infty)$.  We will write $B^s_{p,r}$ for $B^s_{p,r}(\T^d)$ unless the abbreviation could cause confusion.

\subsection{Derivation of the Energy Budget Relation}
Define $E_{\le Q}(s):= \frac{1}{2}\int_{\T^d\times \{s\}} \frac{(\rho u)_{\le Q}^2}{\rho_{\le Q}}\,dx$, the energy associated to length scales $\l_q$ for $q\le Q$.  Defining $U=\frac{(\rho u)_{\le Q}}{\rho_{\le Q}}$ and putting $\psi=U_{\le Q}$ in \eqref{eq:wk_momentum}, we see that 
\begin{equation}
\label{eq:energy_via_wk_momentum}
\begin{split}
2 E_{\le Q}(s)\big|^t_0 = & \int_0^t \int \left( (\rho u)_{\le Q} \cdot \p_s U + (\rho u\otimes u)_{\le Q} : \n U + p_{\le Q}\diver U\right)\,dx\,ds \\ & - \mu\int_0^t \int \n u_{\le Q} : \n U\,dx\,ds + \int_0^t \int (\rho f)_{\le Q}\cdot U.
\end{split}
\end{equation}
On the other hand, we can rewrite the definition of $E_{\le Q}$ using the weak form of the density transport equation.  We apply \eqref{eq:wk_density} in passing from the first to the second line below:
\begin{align*}
E_{\le Q}(s)\big|^t_0 
&=\frac{1}{2}\int_{\T^d\times \{s\}}\rho_{\le Q} U^2\,dx\bigg|^t_0
 =\frac{1}{2}\int_{\T^d\times \{s\}}\rho(U^2)_{\le Q}\,dx\bigg|^t_0\\
&= \frac{1}{2}\int_0^t \int \left( \rho \p_s(U^2)_{\le Q} + (\rho u\cdot \n)(U^2)_{\le Q}\right)\,dx\,ds \\
& = \frac{1}{2}\int_0^t \int \left( \rho_{\le Q} \p_s(U^2) + ((\rho u)_{\le Q}\cdot \n)(U^2)\right)\,dx\,ds \\
& = \int_0^t \int \left( (\rho u)_{\le Q} \cdot \p_s U + ((\rho u)_{\le Q} \otimes U ):\n U \right)\,dx\,ds 
\end{align*}
Subtracting the result from \eqref{eq:energy_via_wk_momentum}, we obtain the energy budget relation at scales $q\le Q$:
\begin{equation}
\label{eq:energy_budget}
E_{\le Q}(t) - E_{\le Q}(0) = \int_0^t \Pi_Q(s)\,ds - \varepsilon_Q(t) + \int_0^t \int (\rho f)_{\le Q} \cdot U\,dx\,ds.
\end{equation}
Here $\Pi_Q(s)$ is the flux through scales of order $Q$ due to the nonlinearity and the pressure, defined by 
\begin{equation} 
\label{eq:energy_flux}
\Pi_Q = \int F_Q(\rho, u):\n U\,dx + \int p_{\le Q} \diver U\,dx,
\end{equation}
\begin{equation}
\label{eq:commutator}
F_Q(\rho, u) = (\rho u \otimes u)_{\le Q} - U\otimes (\rho u)_{\le Q}, 
\end{equation}
and $\varepsilon_Q$ and $\int_0^t \int (\rho f)_{\le Q}\cdot U\,dx\,ds$ represent the energy dissipation due to heat loss and the external force, respectively, at scales $q\le Q$. Now $\varepsilon_Q$ is given by 
\begin{equation*}
\label{eq:heat_dissip_Q}
\varepsilon_Q(t) = \mu \int_0^t \int \n u_{\le Q} : \n U\,dx\,ds.
\end{equation*}
Also denote 
\begin{equation*}
\label{eq:heat_dissip}
\varepsilon(t) = \mu \int_0^t \|\n u\|_2^2\,ds.
\end{equation*}
We aim to show that for appropriate $(\rho, u, p)$ and all $t\in [0,T]$, we have (as $Q\to \infty$) that $E_{\le Q}(t)\to E(t)$, $\int_0^t \Pi_Q(s)\,ds\to 0$, $\varepsilon_Q(t)\to \varepsilon(t)$, and $\int_0^t\int (\rho f)_{\le Q} \cdot U\,dx\,ds \to \int_0^t \int \rho u\cdot f\,dx\,ds$.  These convergences will immediately imply that \eqref{eq:energy_balance} holds for $(\rho, u, p)$.

\subsection{The Localization Kernel and Estimates on the Littlewood-Paley Projections}
Let $a,b\in [1,\infty]$, $s\in (0,1]$, and let $f$ and $g$ be real-valued functions. Define the following:
\[
K^s_q = \left\{ \begin{array}{lcl}
			\l_q^{s-1}, 	 & & q\ge 0; \\
			\l_q^{s},		 & & q<0;
					\end{array}\right.
\hspace{5 mm}
d_{a,q}^s (f) = \l_q^s\|f_q \|_{L^a};
\hspace{5 mm}
D_{a,Q}^s(f) = \sum_{q=-1}^\infty K^s_{Q-q} d^s_{a,q}(f).
\]
We can define these expressions analogously for the vector-valued $f$ and/or $g$.  If both take values in $\R^d$, then the scalar product $fg$ becomes $f\otimes g$.  Note also that in view of summability of the kernel we have 
\begin{equation}
\label{eq:D&d_comp}
\limsup_{Q\ra \infty} D^s_{a,Q}(f) \sim \limsup_{q\to \infty} d^s_{a,q}(f)
\end{equation}
where the similarity constants depend only on $s$. 

\begin{PROP}\label{p:biest}
For $f\in B^s_{a,\infty}$, $g\in B^t_{b,\infty}$, $a,b\in [1,\infty)$, $s,t\in (0,1)$, we have the following estimates:
\begin{align}
\label{ineq:commutator}
	\|(fg)_{\le Q} - f_{\le Q} g_{\le Q} \|_{c}
	& \lesssim \l_Q^{-s-t} D^s_{a,Q}(f) D^t_{b,Q}(g), \quad \frac1c = \frac1a + \frac1b , \\
\label{e:endpoint}
\|(fg)_{\le Q} - f_{\le Q} g_{\le Q} \|_{a} & \lesssim \l_Q^{-s} D_{a,Q}^s(f) \|g\|_\infty\\
\label{ineq:n_f_le}
	\| \n f_{\le Q} \|_a 
		& \lesssim \l_Q^{1 - s} D^s_{a,Q}(f),\\
\label{ineq:>Q}
\|f_{>Q}\|_a & \le \l_Q^{-s} D_{a,Q}^s(f).
\end{align}
\end{PROP}
\begin{remark}
Let us note that \eqref{e:endpoint} is still meaningful when $s=1$. However, in this case, the kernel is not localized in the region $q>0$, which meets finitely many terms in the convolution $D$. Nonetheless, uniform bounds on the convolution would be applicable under stronger summability assumption on Littlewood-Paley components of $f$. For example, when $a = 2$ and $f\in H^1$ we clearly have
\[
D_{2,Q}^1(f) \leq \| f\|_{H^1}.
\]
\end{remark}
\begin{proof}
Since
\[
\widetilde{h}_Q * f = f_{\le Q},
\hspace{5 mm}
\int \widetilde{h}_Q(y)\,dy = 1,
\]
we can write 
\[
(fg)_{\le Q} - f_{\le Q} g_{\le Q} = r_Q(f,g) - f_{>Q}g_{>Q},
\]
where
\begin{equation}
\label{eq:remainder2}
r_Q(f,g) = \int \widetilde{h}_Q(y)(f(\cdot-y)-f(\cdot))(g(\cdot-y)-g(\cdot))\,dy.
\end{equation}
Therefore, to prove \eqref{ineq:commutator} it suffices to estimate $r_Q(f,g)$, $f_{>Q}$, $g_{>Q}$ appropriately.

We can write
\[
\|f_{>Q}\|_a
	\lesssim \l_Q^{-s}\sum_{q > Q} \l_{Q-q}^s \l_q^s \|f_q\|_a = \l_Q^{-s} \sum_{q>Q} K_{Q-q}^s d_{a,q}^s(f)\le \l_Q^{-s} D_{a,Q}^s(f).
\]
The same reasoning yields $\|g_{>Q}\|_b \le \l_Q^{-t} D_{b,Q}^t(g)$, and by H\"older, 
\[
\| f_{>Q}g_{>Q} \|_c \leq \l_Q^{-s-t} D_{a,Q}^s(f)D_{b,Q}^t(g).
\]
Next, we have
\begin{equation}
\label{ineq:f_q_diff}
\|f_q(\cdot - y) - f_q(\cdot)\|_a
	= \left\|\int_0^1 (\n f_q)(\cdot - \theta y)\cdot y\,d\theta \right\|_a 
	\le |y|\|\n f_q\|_a \lesssim |y| \l_q \|f_q\|_a.
\end{equation}
We use \eqref{ineq:f_q_diff} for $q\le Q$ in the following estimate:
\begin{align*}
\|f(\cdot - y) - f(\cdot)\|_a
	& \lesssim \l_Q^{1-s}\sum_{q\le Q} \l_Q^{s-1}\|f_q(\cdot - y) - f_q(\cdot)\|_a + \l_Q^{-s}\sum_{q>Q}\l_Q^s \|f_q(\cdot - y) - f_q(\cdot)\|_a\\
	& \lesssim \l_Q^{1-s}  \sum_{q\le Q}  \l_{Q-q}^{s-1} \l_q^{s-1} \cdot |y| \l_q \|f_q\|_a + \l_Q^{-s} \sum_{q>Q} \l_{Q-q}^s \l_q^s \|f_q\|_a \\
	& = \l_Q^{1-s} |y| \sum_{q\le Q}  K_{Q-q}^s d^s_{a,q}(f) + \l_Q^{-s} \sum_{q>Q} K_{Q-q}^s d^s_{a,q}(f) \\
	& \le (\l_Q |y| + 1)\l_Q^{-s}D_{a,Q}^s(f). 
\end{align*}
Clearly $\|g(\cdot - y) - g(\cdot)\|_b\le (\l_Q |y| + 1 ) \l_Q^{-t} D_{b,Q}^t(g)$, by the same argument.  Now we can easily estimate $r_Q(f,g)$:
\begin{align*}
\| r_Q(f,g) \|_{c}
	& \le \int |\widetilde{h}_Q(y)| \| f(\cdot -y)-f(\cdot)\|_a \|g(\cdot -y)-g(\cdot) \|_b \,dy \\
	& \lesssim \left( \int |\widetilde{h}_Q(y)| (\l_Q |y| + 1)^2 \,dy \right) \l_Q^{-s-t} D_{a,Q}^s(f) D_{b,Q}^t(g)\\
	& \lesssim \l_Q^{-s-t} D_{a,Q}^s(f) D_{b,Q}^t(g).
\end{align*}
This proves \eqref{ineq:commutator}.  The proof of \eqref{e:endpoint} follows the same lines, except we apply
$\| g_{>Q} \|_\infty \leq \|g\|_\infty$, and $\|g(\cdot - y) - g(\cdot)\|_\infty \leq 2 \|g\|_\infty$. The latter results in the term $(\l_Q |y| + 1)$ with power 1 inside the $h_Q$-integral, which is also bounded uniformly in $Q$.

Finally, we write
\begin{align*}
\|\n f_{\le Q} \|_a
	& \lesssim \l_Q^{1-s}\sum_{q\le Q} \l_Q^{s-1} \|\n f_q\|_a 
	\lesssim \l_Q^{1-s} \sum_{q\le Q} \l_{Q-q}^{s-1} \l_q^{s-1} \cdot \l_q \|f_q\|_a \\
	& = \l_Q^{1-s} \sum_{Q-q\ge 0} K_{Q-q}^s d_{a,q}^s(f)
	\le \l_Q^{1-s} D_{a,Q}^s(f).
\end{align*}
\end{proof}

\begin{PROP}
\label{p:Besov_Lp}
Let $f\in B^s_{a,\infty}$, $g\in B^s_{b,\infty}$,  $a,b\in [1,\infty]$, $s\in (0,1)$, $\frac1c=\frac1a + \frac1b$.  Then
\begin{equation}
\|\n(fg)_{\le Q}\|_c \lesssim \l_Q^{1-s} ( D_{a,Q}^s(f) \|g\|_b  + D_{b,Q}^s(g) \|f\|_a ).
\end{equation}
\end{PROP} 
\begin{proof}
First, notice that if $p$ or $r$ is greater than $Q+2$ and $|p-r|>2$, then the Fourier support of $f_p g_r$ lies outside the ball of radius $\l_{Q+1}$ centered at $0$.  In particular, $(f_p g_r)_{\le Q}$ vanishes.  Therefore 
\[
(fg)_{\le Q} = (f_{\le Q+2}g_{\le Q+2})_{\le Q} + \sum_{\substack{\max\{p,r\}>Q+2 \\ |p-r|\le 2}} (f_p g_r)_{\le Q},
\]
so we have
\begin{equation}
\label{ineq:grad_fgQ_decomp}
\|\n (fg)_{\le Q}\|_c \le \|\n (f_{\le Q}g_{\le Q})\|_c + \|\n (f_{\sim Q} g_{\le Q} + f_{\le Q} g_{\sim Q}) \|_c + \sum_{\substack{ p,r>Q \\ |p-r|\le 2 }} \|\n (f_p g_r)_{\le Q}\|_c.
\end{equation}
We estimate each of the terms on the right side of this inequality.  First, we have 
\begin{equation}
\label{ineq:grad_fQgQ}
\|\n (f_{\le Q} g_{\le Q})\|_c 
	\le \|\n f_{\le Q}\|_a \|g\|_b + \|\n g_{\le Q}\|_b \|f\|_a \\
	\lesssim \l_Q^{1-s} \big( D_{a,Q}^s(f) \|g\|_b + D_{b,Q}^s(g) \|f\|_a \big).
\end{equation}
Next, 
\begin{align*}
\|\n (f_{\sim Q}g_{\le Q})\|_c
& \lesssim \|\n f_{\sim Q}\|_a \|g\|_b + \|\n g_{\le Q}\|_b \|f\|_a\\
& \lesssim \l_Q^{1-s} \big( D_{a,Q}^s(f) \|g\|_b + D_{b,Q}^s(g) \|f\|_a\big),
\end{align*}
where we note that  $\|\n f_{\sim Q}\|_a \sim \l_Q^{1-s}D_{a,Q}^s(f)$ and use \eqref{ineq:n_f_le} in order to obtain the second inequality.
We can estimate $\|\n(f_{\le Q}g_{\sim Q})\|_c$ similarly, concluding that 
\begin{equation}
\label{ineq:grad_fsimgQ}
\begin{split}
\|\n (f_{\sim Q} g_{\le Q} + f_{\le Q} g_{\sim Q}) \|_c
\lesssim \l_Q^{1-s} ( D_{a,Q}^s(f) \|g\|_b + D_{b,Q}^s(g) \|f\|_a ).
\end{split}
\end{equation}
By differential Bernstein's and H\"older inequalities we have
\[
\|\n (f_p g_r)_{\le Q}\|_c
	\lesssim \l_Q \|f_p\|_a \|g_r\|_b
\]
Using this we obtain 
\begin{align*}
\sum_{\substack{ p,r>Q \\ |p-r|\le 2 }} \|\n (f_p g_r)_{\le Q}\|_c
& \lesssim \l_Q^{1-s} \sum_{p>Q}\l_{Q-q}^s \l_q^s \|f_p\|_a \|g\|_b \le \l_Q^{1-s} D_{a,Q}^s(f)\|g\|_b.
\end{align*}
Combining this estimate with \eqref{ineq:grad_fgQ_decomp}, \eqref{ineq:grad_fQgQ}, and \eqref{ineq:grad_fsimgQ} immediately yields the desired statement.
\end{proof}
\begin{remark}\label{r:pressure}
One can also show (by a proof nearly identical to the above) that if $f,g\in B_{a,\infty}^s\cap L^b$ with $\frac1a + \frac1b=\frac1c$ and $a,b\in [1,\infty]$, then $\|\n(fg)_{\le Q}\|_c \lesssim \l_Q^{1-s} ( D_{a,Q}^s(f) \|g\|_b  + D_{a,Q}^s(g) \|f\|_b )$.
\end{remark}
\begin{remark}
Recall the following result for the classical Navier-Stokes equations (i.e. \eqref{eq:momentum} and \eqref{eq:incompressibility}, with $\rho\equiv 1$, $f\equiv 0$): If $(u,p)$ is a weak solution, with $u\in C^\alpha$ for some $\alpha\in (0,1)$, then $p=\Delta^{-1}(\diver \diver(u\otimes u))\in C^\alpha$.  We can generalize this result using Proposition~\ref{p:Besov_Lp}: Assume $u\in B^s_{a,\infty}$, with $a\in [2,\infty]$ and $s\in (0,1)$; then $p\in B^s_{a/2,\infty}$.  Indeed, we have 
\[
\l_Q^s \|p_Q\|_{a/2} \sim \l_Q^{-(1-s)}\|\diver (u\otimes u)_Q\|_{a/2} \lesssim (D_{a,Q}^s(u))^2.
\]
This observation motivates our integrability assumption on $p$ in Theorem~\ref{t:energy_bal}.
\end{remark}

\section{Estimates on the Flux}

First, we give a decomposition of $F_Q(\rho, u)$ which is more conducive to estimates.  In order to do so we define, in analogy with \eqref{eq:remainder2}, the quantity
\[
r_Q(\rho, u,u)
= \int \widetilde{h}_Q(y) [\rho(x-y)-\rho(x)][u(x-y)-u(x)]\otimes [u(x-y)-u(x)]\,dy.
\]
\begin{lemma}
$F_Q(\rho, u)$ can be written as
\begin{equation}
\label{eq:commutator_alt}
\begin{split}
F_Q(\rho, u) = & r_Q(\rho, u, u) - \frac{1}{\rho_{\le Q}}[(\rho u)_{\le Q} - \rho_{\le Q} u_{\le Q}]\otimes [(\rho u)_{\le Q} - \rho_{\le Q} u_{\le Q}] + \rho_{> Q}u_{>Q}\otimes u_{>Q} \\ & + 2 Sym([(\rho u)_{\le Q} - \rho_{\le Q} u_{\le Q}] \otimes u_{>Q}) + \rho[(u\otimes u)_{\le Q} - u_{\le Q} \otimes u_{\le Q}].
\end{split}
\end{equation}
\end{lemma}
\begin{proof}
We can write
\begin{align*}
r_Q(\rho, u,u)
	& = (\rho u\otimes u)_{\le Q} - 2Sym[(\rho u)_{\le Q} \otimes u] + \rho_{\le Q}u\otimes u -\rho r_Q(u,u)\\
	& =(\rho u\otimes u)_{\le Q} - 2Sym[((\rho u)_{\le Q}-\rho_{\le Q}u_{\le Q}) \otimes u] \\
	& \hspace{5 mm} + \rho_{\le Q} (u\otimes u - u_{\le Q} \otimes u - u\otimes u_{\le Q}) - \rho r_Q(u,u)\\
	& =(\rho u\otimes u)_{\le Q} - 2Sym[((\rho u)_{\le Q}-\rho_{\le Q}u_{\le Q}) \otimes u] \\
	& \hspace{5 mm} - \rho_{\le Q}u_{\le Q} \otimes u_{\le Q} + \rho_{\le Q} u_{>Q}\otimes u_{>Q}  -  \rho r_Q(u,u)\\ 
	& =(\rho u\otimes u)_{\le Q} - 2Sym[((\rho u)_{\le Q}-\rho_{\le Q}u_{\le Q}) \otimes u] - \rho_{\le Q}u_{\le Q} \otimes u_{\le Q} \\
	& \hspace{5 mm} - \rho[(u\otimes u)_{\le Q} - u_{\le Q} \otimes u_{\le Q}] - \rho_{>Q}u_{>Q}\otimes u_{>Q}, 
\end{align*}
where $Sym$ denotes the symmetric part. Therefore
\begin{align*}
(\rho u\otimes u)_{\le Q}
	& = r_Q(\rho, u, u) + 2 Sym([(\rho u)_{\le Q} - \rho_{\le Q} u_{\le Q}] \otimes u) \\ 
	& \hspace{5 mm} + \rho [ (u\otimes u)_{\le Q} - u_{\le Q} \otimes u_{\le Q}] + \rho_{\le Q}u_{\le Q} \otimes u_{\le Q} + \rho_{> Q}u_{>Q}\otimes u_{>Q}. \\
\end{align*}
Since we also have 
\begin{align*}
\frac{(\rho u)_{\le Q} \otimes (\rho u)_{\le Q}}{\rho_{\le Q}}
	& = \frac{1}{\rho_{\le Q}}[(\rho u)_{\le Q} - \rho_{\le Q} u_{\le Q}]\otimes [(\rho u)_{\le Q} - \rho_{\le Q} u_{\le Q}] \\
	& \hspace{5 mm} + 2Sym([(\rho u)_{\le Q} - \rho_{\le Q} u_{\le Q}]\otimes u_{\le Q}) + \rho_{\le Q} u_{\le Q}\otimes u_{\le Q},
\end{align*}
subtracting the right sides of the last two equations gives the desired representation.
\end{proof}

\begin{theorem}
\label{t:flux_to0}
Assume that $0< \underline{\rho}\le \rho \le \overline{\rho} <\infty$ and that $(\rho, u, p)$ satisfies 
\begin{equation}
\label{hyp:Onsager2space}
\rho \in B_{a,\infty}^{1/3},\;
u\in B_{b,c_0}^{1/3},\;
p\in B^{1/3}_{b/2,\infty}, \quad \frac1a + \frac3b=1,\; a\ge 2.
\end{equation}
Then the flux $\Pi_Q$ defined by \eqref{eq:energy_flux} tends to zero as $Q\to \infty$. 
\end{theorem}
\begin{proof}

Clearly
\[
\|r_Q(\rho, u, u)\|_{b/2} \lesssim \int |\widetilde{h}_Q(y)| \| u(\cdot -y)-u(\cdot)\|_b^2 \,dy,
\]
and we can follow the proof of Proposition~\ref{p:biest} to conclude $\|r_Q(\rho, u, u)\|_{b/2} \lesssim \l_Q^{-2/3}(D_{b,Q}^{1/3}(u))^2$.

Using \eqref{e:endpoint}, we can estimate 
\[
\left\| \frac{1}{\rho_{\le Q}}[(\rho u)_{\le Q} - \rho_{\le Q} u_{\le Q}]\otimes [(\rho u)_{\le Q} - \rho_{\le Q} u_{\le Q}]\right\|_{b/2}
\lesssim \underline{\rho}^{-1}(\l_Q^{-1/3} D_{b,Q}^{1/3}(u) \overline{\rho} )^2 \lesssim \l_Q^{-2/3} (D_{b,Q}^{1/3}(u))^2
\]
Using $\|\rho_{>Q}\|\le \overline{\rho}$ and \eqref{ineq:>Q}, we get
\[
\|\rho_{>Q} u_{>Q}\otimes u_{>Q}\|_{b/2} \lesssim \l_Q^{-2/3}(D_{b,Q}^{1/3}(u))^2.
\]
Combining \eqref{e:endpoint} and \eqref{ineq:>Q} yields 
\[
\| [(\rho u)_{\le Q} - \rho_{\le Q} u_{\le Q}] \otimes u_{>Q} \|_{b/2} \le (\l_Q^{-1/3} D_{b,Q}^{1/3}(u)\overline{\rho})(\l_Q^{-1/3}D_{b,Q}^{1/3}(u)) \lesssim \l_Q^{-2/3} (D_{b,Q}^{1/3}(u))^2
\]
Finally, 
\[
\|\rho[(u\otimes u)_{\le Q} - u_{\le Q}\otimes u_{\le Q}]\|_{b/2} \lesssim \overline{\rho} \l^{-2/3}(D_{b,Q}^{1/3}(u))^2 \lesssim \l^{-2/3}(D_{b,Q}^{1/3}(u))^2.
\]
Therefore, 
\[
\|F_Q(\rho, u)\|_{b/2} \lesssim \l_Q^{-2/3}(D_{b,Q}^{1/3}(u))^2.
\]

We also have $\n U = \rho_{\le Q}^{-1} \n(\rho u)_{\le Q} - \rho_{\le Q}^{-2} (\rho u)_{\le Q} \otimes \n \rho_{\le Q}$.  Write $\frac1a+\frac1b=\frac1c$.  Then using the two Propositions of the previous section, we estimate:
\begin{align*}
\|\n U\|_c
	& \lesssim \|\n (\rho u)_{\le Q}\|_c + \|\rho u\|_b \|\n \rho_{\le Q}\|_a 
\lesssim \l_Q^{2/3} \big( D_{a,Q}^{1/3}(\rho) \|u\|_b + D_{b,Q}^{1/3}(u) \big).
\end{align*}
Therefore 
\begin{equation}
\int F_Q(\rho, u):\n U dx
	\lesssim (D_{b,Q}^{1/3}(u))^2 \big( D_{a,Q}^{1/3}(\rho) \|u\|_b + D_{b,Q}^{1/3}(u) \big)
\end{equation}

Next, we deal with the pressure term.  Note that by \eqref{eq:wk_incompressibility}, we have 
\[
\int p_{\le Q} \diver U\,dx = -\int \n p_{\le Q} \cdot (U - u_{\le Q})\,dx.
\]
So 
\begin{align*}
\int_{\T^d} p_{\le Q} \diver U dx
	& \lesssim \| \n p_{\le Q} \|_{b/2} \|(\rho u)_{\le Q} - \rho_{\le Q}u_{\le Q} \|_c \\
	& \lesssim \l_Q^{2/3} D_{b/2,Q}^{1/3}(p) \cdot \l^{-2/3} D_{a,Q}^{1/3}(\rho) D_{b,Q}^{1/3}(u)
	= D_{a,Q}^{1/3}(\rho) D_{b,Q}^{1/3}(u) D_{b/2,Q}^{1/3}(p).
\end{align*}
Thus 
\begin{equation*}
|\Pi_Q|
	\lesssim D_{b,Q}^{1/3}(u)\bigg[ D_{b,Q}^{1/3}(u) \big(D_{a,Q}^{1/3}(\rho)\|u\|_b + D_{b,Q}^{1/3}(u) \big) + D_{a,Q}^{1/3}(\rho) D_{b/2,Q}^{1/3}(p)\bigg].
\end{equation*}
In view of \eqref{eq:D&d_comp} and our assumptions on $\rho, u, p$, the bracketed term in each estimate is uniformly bounded in $Q$, while $D_{b,Q}^{1/3}(u)$ tends to zero as $Q\to \infty$.  Therefore $\lim_{Q\to \infty} \Pi_Q = 0$, as claimed.
\end{proof}

Note that we obtain Theorem~\ref{t:EE_result} as a Corollary: By Theorem~\ref{t:flux_to0}, as well as  \eqref{eq:D&d_comp} and the Dominated Convergence Theorem, we have
\[
E_{\le Q}(t) - E_{\le Q}(0) = 
\int_0^t \Pi_Q(s)\,ds \stackrel{Q\to \infty}{\longrightarrow} 0.
\]
Now we prove Theorem~\ref{t:energy_bal}:
\begin{proof}
As noted above, we have $\int_0^t \Pi_Q(s)\,ds \to 0$.
It remains to show $\varepsilon_Q(t)\to \varepsilon(t)$ and $\int_0^t\int (\rho f)_{\le Q} \cdot U\,dx\,ds \to \int_0^t \int \rho u\cdot f\,dx\,ds$. So, let us make the following observation:
\[
\int \n u_{\leq Q} : \n U dx = \int \n u_{\leq Q} : \n (U - u_{\leq Q}) dx + \| \n u_{\leq Q}\|_2^2.
\]
Clearly, 
\[
\int_0^t  \| \n u_{\leq Q}(s)\|_2^2 ds \ra \int_0^t  \| \n u(s)\|_2^2 ds.
\]
Next, 
\[
\int \n u_{\leq Q} : \n (U - u_{\leq Q}) dx = -\int \D u_{\leq Q} : ((\rho u)_{\leq Q} - \rho_{\le Q} u_{\leq Q} ) \rho_{\leq Q}^{-1} dx.
\]
Using \eqref{e:endpoint} and the remark following Proposition~\ref{p:biest} we estimate
\[
\left| \int \D u_{\leq Q} : ((\rho u)_{\leq Q} - \rho_{\leq Q} u_{\leq Q} ) \rho_{\leq Q}^{-1} dx \right| \leq \| \D u_{\leq Q} \|_2 \l_Q^{-1} \| u\|_{H^1} \|\rho\|_\infty.
\]
Then 
\[
\| \D u_{\leq Q} \|_2 \l_Q^{-1} \leq \left( \sum_{q \leq Q} \l_{q-Q}^2 \| \n u_q\|_2^2 \right)^{1/2}.
\]
Since the latter vanishes as $Q \ra \infty$ a.e. in time and is uniformly bounded by the dominant $H^1$-norm of $u$ we obtain
\[
\int_0^t \| \D u_{\leq Q} \|_2 \l_Q^{-1} \| u\|_{H^1} ds \leq \|u\|_{L^2H^1} \left(\int_0^t  \sum_{q \leq Q} \l_{q-Q}^2 \| \n u_q\|_2^2  ds \right)^{1/2}\ra 0.
\]
Finally, the convergence $\int_0^t\int (\rho f)_{\le Q} \cdot U\,dx\,ds \to \int_0^t \int \rho u\cdot f\,dx\,ds$ is rather straightforward. Indeed, write
\[
\rho f \cdot u  - (\rho f)_{\le Q}\cdot U
	= (\rho f - (\rho f)_{\le Q}) u + (\rho f)_{\le Q}(u - U).
\]	
Note that $u, \rho f \in L^2_{t,x}$, hence $(\rho f)_{\le Q} \ra \rho f$ strongly in $L^2_{t,x}$, and hence $\int (\rho f - (\rho f)_{\le Q}) u \ra 0$. Similarly, $u - U = \frac{1}{\rho_{\leq Q}} (\rho_{\leq Q} u - (\rho u)_{\leq Q})=\frac{1}{\rho_{\leq Q}} (\rho_{\leq Q} u_{\leq Q} - (\rho u)_{\leq Q})  + u_{>Q}$. Again, $u_{>Q} \ra 0$ in $L^2_{t,x}$, while for the difference $\rho_{\leq Q} u_{\leq Q} - (\rho u)_{\leq Q}$ we can use \eqref{e:endpoint} with $s=1$, $a=2$ to conclude that it also tends to zero in $L^2_{t,x}$. This finishes the proof.
 
\end{proof}

\begin{remark} Let us discuss a few extensions. First,
one can see from the proof that the full strength of the integrability in time assumption on $u$ was not used. Rather, the hypothesis $u\in L^b(0,T;B^{1/3}_{b,\infty})$ can be replaced by the weaker assumption that 
\[
\lim_{q \to 0} \int_0^T \l_q^{b/3}\|u_q\|_b^b\,ds = 0.
\]
This is equivalent to a space-time averaged increment condition
\[
\lim_{y \ra 0} \frac{1}{|y|^{b/3}} \int_{\T^d \times [0,T]} |u(x+y,t) - u(x,t)|^b dx dt = 0.
\]
Second, time integrability in \eqref{hyp:Onsager2} can be replaced with its own exponents
\[
\rho\in L^{a'}B_{a,\infty}^{\frac13},\;
u\in L^{b'}B_{b,c_0}^{\frac13},\;
p\in L^{\frac{b'}{2}}B^{\frac13}_{\frac{b}{2},\infty},\quad \frac1a + \frac3b=1,\; \frac{1}{a'} + \frac{3}{b'}=1.
\]
Finally, it appears possible to extend the results to the system with density-dependent kinematic viscosity $\mu = \mu(\rho)$ with sufficiently smooth $\mu$. We leave calculations pertaining to this case to future research.
\end{remark}

%\bibliographystyle{plain}
%\bibliography{khm}

\end{document}